\documentclass{aims}
\usepackage{amsmath}
  \usepackage{paralist}
  \usepackage{graphics} 
  \usepackage{epsfig} 
\usepackage{graphicx}  \usepackage{epstopdf}
 \usepackage[colorlinks=true]{hyperref}
\hypersetup{urlcolor=blue, citecolor=red}

\def\currentvolume{X}
 \def\currentissue{X}
  \def\currentyear{200X}
   \def\currentmonth{XX}
    \def\ppages{X--XX}
     \def\DOI{10.3934/xx.xx.xx.xx}

\def\d{{\rm d}}

\def\R {\mathbb{R}}

\def\L {{\mathcal K}}

\def\AA {{\mathcal A}}

\def \l {\langle}
\def \r {\rangle}

\newtheorem{theorem}{Theorem}[section]

\newtheorem{lemma}[theorem]{Lemma}
\newtheorem{proposition}{Proposition}[section]

\theoremstyle{definition}

\title[Propagation of singularities for Hamilton-Jacobi equations] 
      {Global Propagation of Singularities\\for Time Dependent Hamilton-Jacobi Equations}

\author[Piermarco Cannarsa, Marco Mazzola and Carlo Sinestrari]{}


 \email{email1@smsu.edu}
 \email{marco.mazzola@imj-prg.fr}
 \email{email3@ece.pdx.edu}

\thanks{The first author is supported by NSF grant xx-xxxx}

\begin{document}
\maketitle

\centerline{\scshape Piermarco Cannarsa}
\medskip
{\footnotesize
 \centerline{Dipartimento di Matematica}
   \centerline{Universit\`a di Roma ``Tor Vergata''}
   \centerline{ Via della Ricerca Scientifica 1, 00133 Roma, Italy}
} 

\medskip

\centerline{\scshape Marco Mazzola}
\medskip
{\footnotesize
 \centerline{CNRS, IMJ-PRG, UMR 7586}
 \centerline{Sorbonne Universit\'es, UPMC Univ Paris Diderot, Sorbonne Paris Cit\'e}
   \centerline{Case 247, 4 Place Jussieu, 75252 Paris, France}
}

\medskip

\centerline{\scshape Carlo Sinestrari}
\medskip
{\footnotesize
 \centerline{Dipartimento di Matematica}
   \centerline{Universit\`a di Roma ``Tor Vergata''}
   \centerline{ Via della Ricerca Scientifica 1, 00133 Roma, Italy}
   }

\bigskip

 \centerline{(Communicated by the associate editor name)}

\begin{abstract}
We investigate the properties of the set of singularities of semiconcave solutions of Hamilton-Jacobi equations of the form
\begin{equation}\label{abstract:EQ}
    u_t(t,x)+H(\nabla u(t,x))=0, \qquad\text{a.e. }(t,x)\in (0,+\infty)\times\Omega\subset\R^{n+1}\,.
\end{equation}
It is well known that the singularities of such solutions propagate locally along generalized characteristics. Special generalized characteristics, satisfying an energy condition, can be constructed, under some assumptions on the structure of the Hamiltonian $H$. In this paper, we provide estimates of the dissipative behavior of the energy along such curves. As an application, we prove that the singularities of any viscosity solution of (\ref{abstract:EQ}) cannot vanish in a finite time.
\end{abstract}

\section{Introduction}
It is commonly accepted that, in  optimal control, a crucial role is played by the Hamilton-Jacobi equation
\begin{equation}\label{intro:HJ}
\begin{cases}
\hspace{.cm}
u_t(t,x)+H\big(x,\nabla u(t,x)\big)=0 
&(t,x)\in (0,T)\times\R^n
\\
u(0,x)=u_0(x)
&x\in\R^n
\hspace{.cm}
\end{cases}
\end{equation}
where
\begin{itemize} 
\item $H:\R^n\times \R^n\to \R$ is a $C^2$ smooth function such that
 \begin{eqnarray*}
&(a)&
\lim_{|p|\rightarrow\infty}\inf_{x\in \R^n}\frac{H(x,p)}{|p|}=+\infty
\\
&(b)&
D_p^2H(x,p)
>0,\quad\forall (x,p)\in\R^n\times \R^n
\end{eqnarray*}
\item 
$ u_0:\R^n\to\R$ is a Lipschitz function.
\end{itemize}
Indeed, being able to characterize the value function as the unique solution of \eqref{intro:HJ} is the starting point towards a rigorous approach to dynamic programing. 

The notion of viscosity solutions, introduced in the seminal papers \cite{CL} and \cite{CEL}, provides the right class of generalized solutions to study existence, uniqueness, and stability issues for problem \eqref{intro:HJ}. The reader will find an overview of the main features of this theory in \cite{BCD}, for first order equations, and \cite{FS}, for second order problems.

On the other hand, it is also well known that Hamilton-Jacobi equations have no global smooth solutions, in general, because solutions may develop singularities---i.e., discontinuities of the gradient---in finite time due to crossing of characteristics. 

Indeed, the maximal regularity one may expect for solutions of \eqref{intro:HJ} is that, for any $t>0$, $u(t,\cdot)$ is locally semiconcave on $\R^n$, that is, $u(t,\cdot)$ can be represented as the sum of a concave and a smooth function on each compact subset of $\R^n$.
In fact, the notion of semiconcave solution was used in the past even to provide existence and uniqueness results for \eqref{intro:HJ}  before the theory of viscosity solution was developed, see \cite{Do}, \cite{Kruzkhov}, and \cite{Krylov}. 
Nowadays, semiconcavity is still an important property in the study of Hamilton-Jacobi when related to optimal control problems in euclidean spaces (\cite{cs}) and even on Riemannian manifolds (\cite{Villani}). However, it is rather regarded as a regularizing effect of the nonlinear semigroup associated with  \eqref{intro:HJ}---in some sense, a sign of irreversibility in Hamilton-Jacobi dynamics.

Another evidence of irreversibility for 
the equation
 \begin{equation}\label{intro:E}
u_t(t,x)+H\big(x,\nabla u(t,x)\big)=0 
\qquad(t,x)\in (0,T)\times\mathbb{R}^n
\end{equation}
is the {\em persistence of singularities}, that is, the fact that once a singularity is created, it will propagate forward in time up to $+\infty$.
Unlike the gain of semiconcavity, such a phenomenon is not well understood so far.

What is sufficiently clear to this date is the local propagating structure of the singular set of a viscosity solution $u$ of \eqref{intro:E}:
if $(t_0,x_0)\in[0,+\infty)\times\R^n$ is a singular point of $u$, then there exists a Lipschitz arc $\gamma:[t_0,t_0+\tau)\to\R^n$ such that
$(t,\gamma(t))$ is singular for all $t\in[t_0,t_0+\tau)$, see \cite{AC2}, \cite{Y}, and \cite{CY}. Therefore, the question we are now interested in is to provide conditions to ensure that $\tau=+\infty$. We note that, in general, this problem has a negative answer if $H$ is allowed Lipschitz dependence in $(t,x)$ even for $n=1$, see \cite[Example~5.6.7]{cs}.

A first, simple case where the answer to the above problem is positive is when $n=1$ and $H$ is sufficiently smooth. Indeed, the $x$-derivative of $u$ turns out to be a solution of a conservation law for which the results in \cite{D} ensure the persistence of singularities. However, the one-dimensional case is very special because topological obstructions prevent singularities from disappearing after their onset.

Another result that guarantees the global propagation of singularities in any space dimension was obtained in \cite{ac} for {\em concave}  solutions  of the Hamilton-Jacobi equation $$u_t+H(\nabla u)=0.$$ Requiring concavity---unlike semiconcavity---for $u$ is, however, a restrictive assumption because it imposes a global constraint on solutions. 

A last result, which is strongly related to the above problem, concerns the distance function, $d_\Omega$, from the boundary of a bounded open subset $\Omega$ of a Riemannian manifold. Such a function is indeed the solution of a well known stationary Hamilton-Jacobi equation, that is, the eikonal equation. In \cite{acns}, it is shown that the singular set of $d_\Omega$ is invariant under the generalized gradient flow, a property which is crucial to prove that $\Omega$ has the same homotopy type as the singular set of $d_\Omega$.

In this paper, we address the above problem for solution of the Hamilton-Jacobi equation
\begin{equation}\label{intro:EQ}
    u_t(t,x)+H(\nabla u(t,x))=0, \qquad(t,x)\in (0,+\infty)\times\Omega,
\end{equation}
where $\Omega$ is a bounded domain in $\R^n$ and $H:\R^n\to\R$ is the quadratic form $$H(p)={1\over2}Ap\cdot p,$$ with $A$ a positive definite $n\times n$ real matrix. Our main result, Theorem~\ref{th:persistence}, ensures that the singularities of any viscosity solution of \eqref{intro:EQ} persist for all times, or at least until the singular arc touches the boundary of $\Omega$. More precisely, we show that, if $(t_0,x_0)$ is a singular point of $u$, then there exist $T\in(0,+\infty]$ and a Lipschitz continuous arc $\gamma:[t_0,t_0+T)\to\R^n$, starting from $x_0$, with
$(s,\gamma(s))$ singular for all $s\in[t_0,t_0+T)$
and such that $$\lim_{s\to t_0+T}\gamma(s)\in\partial \Omega$$whenever $T<+\infty$.

The proof of the above result relies on two main ideas that are converted in two technical results, respectively. In the first one, Lemma~\ref{lm},  we obtain, as in \cite{acns}, a sharp semiconcavity estimate for a suitable transform of the solution $u$. In the second one, Theorem~\ref{th}, we establish an inequality showing that the full Hamiltonian associated with \eqref{intro:EQ}, that is,
$$F(\tau,p)=\tau+H(p),$$ 
decreases along a selection of the superdifferential of $u$, evaluated at any point of a suitable arc. 
To be more specific,  let $(t_0,x_0)\in Q$ and let $\bar t< t_0$. Then we prove that there exist $T'>0$ and a Lipschitz continuous arc $\gamma:[t_0,t_0+T')\to\Omega$, starting from $x_0$,  such that
\begin{equation}\label{eq:introdissi}
\min_{(\tau,p)\in D^+u(s,\gamma(s))}F(\tau,p)\leq\left(t_0-\bar t\over s-\bar t\right)^2\min_{(\tau_0,p_0)\in D^+u(t_0,x_0)}F(\tau_0,p_0)
\end{equation}
for every $s\in[t_0,t_0+T')$. Such a dissipative behavior is essential to deduce persistence of singularities. Indeed, the above inequality yields that, if $(t_0,x_0)$ is a singular point of $u$, hence the minimum on the right-hand side is strictly negative, then the minimum on the left side must be negative too, thus forcing  the point $(s,\gamma(s))$ to be singular as well for every $s\in[t_0,t_0+T')$. Moreover, the quantitative estimate  \eqref{eq:introdissi} allows to reproduce the above reasoning starting from $t_0+T'$ as long as the arc $\gamma$ stays away from the boundary of $\Omega$.

Although the structure of the Hamiltonian in \eqref{intro:EQ} is quite special, we believe that our approach can be used to treat more complex classes of equations. In a future work, we will show how to adapt the above ideas to treat time dependent Hamilton-Jacobi equations on Riemannian manifolds.

\section{Preliminaries}

Let $\Omega\subset\R^n$ be an open set. We consider the Hamilton-Jacobi equation
\begin{equation}\label{syst}
\left\{\begin{array}{ll}
    u_t(t,x)+H(\nabla u(t,x))=0 & \mbox{ a.e. }(t,x)\in (0,+\infty)\times\Omega=:Q\\
    u(t,x)=\varphi(t,x) & \mbox{ for }(t,x)\in \partial Q,
\end{array}\right.
\end{equation}
where $H:\R^n\to\R$ is the quadratic form $$H(p)={1\over2}Ap\cdot p,$$ with $A$ positive definite, and $\varphi:\overline Q\to\R$ is Lipschitz continuous.
Here we define $$u_t={\partial u\over \partial t}\qquad \mbox{and}\qquad\nabla u=\left({\partial u\over \partial x_1},\ldots,{\partial u\over \partial x_n}\right),$$ 
whereas $D u=(u_t,\nabla u)$ indicates the gradient of $u$ whenever it exists.

Let $L$ denote the Legendre transform of $H$, i.e. $L(q)={1\over2}A^{-1}q\cdot q$. We assume that the data $\varphi$ satisfy the following compatibility condition
\begin{equation}\label{comp}
\varphi(t,x)-\varphi(s,y) \leq (t-s) L\left(\frac{x-y}{t-s}\right),
\end{equation}
for all $(t,x), (s,y) \in \partial Q$ such that $t >s \geq 0$.

Then, see \cite{L}, problem \eqref{syst} has a unique viscosity solution $u \in \rm{Lip}(\bar Q)$ which is given by the Hopf formula
\begin{equation}\label{hopf}
u(t,x)=\min_{\begin{array}{cc}(s,y)\in\partial Q\\s<t\end{array}}\left[(t-s)L\left({x-y\over t-s}\right)+\varphi(s,y)\right].
\end{equation}
%
%
Equivalently, see \cite{cs}, $u$ given by \eqref{hopf} is the unique Lipschitz function on $\bar Q$ which satisfies \eqref{syst} almost everywhere and is locally semiconcave on $Q$, that is, for every convex compact $K\subset Q$ there is a constant $C\geq0$ such that
$$u(X+H)+u(X-H)-2u(X)\leq C|h|^2$$
for every $X,H\in\R^{n+1}$ such that $X+H,X-H\in K$.
%
%
%
%
%

Since the function $u$ given by \eqref{hopf} is in general non differentiable, we denote by $\Sigma(u)$ the set of points of non-differentiability for $u$. For $X\in Q$ we introduce the superdifferential of $u$ at $X$
$$D^+u(X)=\left\{P\in\R^{n+1}\,:\,\limsup_{Y\to X}{u(Y)-u(X)-\l P,Y-X\r\over|Y-X|}\leq0\right\}$$
and the set of the reachable gradients of $u$ at $X$
$$D^*u(X)=\left\{P\in\R^{n+1}\,:\,Q\setminus\Sigma(u)\ni X_i\to X,\ Du(X_i)\to P\right\}.$$
The directional derivative of $u$ at $X$ in the direction $V\in\R^{n+1}$ is defined by
$$\partial u(X,V)=\lim_{h\to0+}{u(X+hV)-u(X)\over h}$$
and the exposed face of $D^+u(X)$ in the direction $V$ by
$$D^+u(X,V)=\left\{P\in D^+u(X)\,:\,\l P,V\r\leq \l Q,V\r \quad\forall\,Q\in D^+u(X)\right\}.$$

Since $u$ is a locally semiconcave map, the set $D^+u(X)$ is the convex hull of $D^*u(X)$. Moreover, see \cite{cs}, if $K$ is a convex compact subset of $Q$ and $C\geq0$ is a semiconcavity constant of $u$ on $K$, the superdifferential of $u$ satisfies the following monotonicity property:
\begin{equation}\label{monsup}
\l P-Q,X-Y\r\leq C|X-Y|^2
\end{equation}
for every $P\in D^+u(X)$, $Q\in D^+u(Y)$ and $X,Y\in K$.
The directional derivative of the locally semiconcave map $u$ can be connected to the superdifferential and the reachable gradients in the following way:
\begin{equation}\label{connect}
\partial u(X,V)=\min_{P\in D^+u(X)}\l P,V\r=\min_{P\in D^*u(X)}\l P,V\r
\end{equation}
for any $X\in Q$ and $V\in\R^{n+1}$. Finally, we recall a result involving the exposed face of $u$. Its proof can be found in \cite{cs}.
\begin{proposition}\label{exposed}
Let $X\in Q$ and $P\in\R^{n+1}$. Suppose that there are sequences $\{X_i\}\subset Q\setminus\{X\}$ and $P_i\in D^+u(X_i)$ satisfying
$$X_i\to X,\qquad P_i\to P\qquad\text{and}\qquad\lim_{i\to\infty}{X_i-X\over|X_i-X|}=V\,.$$
Then $P\in D^+u(X,V)$.
\end{proposition}

\section{A sharp monotonicity estimate}

We first prove a technical lemma. Here, for every $X=(t,x)\in\R^{n+1}$, $0<\bar t<t$ and $R>0$, we denote by $B(X,R)$ the ball of centre $X$ and radius $R$, while by $B_{\bar t}(X,R)$ we refer to the set
$$B_{\bar t}(X,R)=((\bar t,+\infty)\times\R^n)\cap B(X,R).$$

\begin{lemma}\label{lm1}
For every $X'=(t',x')\in Q$ there exist $R>0$ and $0<\bar t<t'$ such that $\overline{B_{\bar t}(X',R)}\subset Q$ and
\begin{equation}\label{t0}
u(t,x)=\min_{y\in\Omega}\left[(t-\bar t)L\left({x-y\over t-\bar t}\right)+u(\bar t,y)\right]
\end{equation}
for each $(t,x)\in B_{\bar t}(X',R)$.
\end{lemma}
\begin{proof} Let $\bar t>0$ and set $Q_{\bar t}=(\bar t,+\infty)\times\Omega$. First we observe that
\begin{equation}\label{hopft0}
u(t,x)=\min_{\begin{array}{cc}(s,y)\in\partial Q_{\bar t}\\s<t\end{array}}\left[(t-s)L\left({x-y\over t-s}\right)+u(s,y)\right]
\end{equation}
for every $(t,x)\in Q_{\bar t}$. Indeed, let $(t,x),(s,y)\in\partial Q_{\bar t}$ such that $t>s\geq\bar t$. Since $t>\bar t$ and $(t,x)\in\partial Q_{\bar t}$, we have that $x\in\partial\Omega$. If $y\in\partial\Omega$, by (\ref{comp}) we obtain
$$u(t,x)-u(s,y)=\varphi(t,x)-\varphi(s,y) \leq (t-s) L\left(\frac{x-y}{t-s}\right)\,.$$
Otherwise, let $y\in\Omega$, so that $s=\bar t$, and let $(r,z)\in\partial Q$ such that $r<s$ and
\begin{equation}\label{joife}
u(s,y)=(s-r) L\left(\frac{y-z}{s-r}\right)+\varphi(r,z)\,.
\end{equation}
Then, by (\ref{joife}), (\ref{comp}) and the structure of $L$, we have
\begin{equation*}
\begin{split}
&u(t,x)-u(s,y)\\
=&\,\varphi(t,x)-(s-r) L\left(\frac{y-z}{s-r}\right)-\varphi(r,z)\\
\leq&\,(t-r) L\left(\frac{x-z}{t-r}\right)-(s-r) L\left(\frac{y-z}{s-r}\right)\\
=&\,(t-r) L\left(\frac{x-y}{t-r}\right)+{1\over t-r}A^{-1}(x-y)\cdot(y-z)+(t-r) L\left(\frac{y-z}{t-r}\right)-(s-r) L\left(\frac{y-z}{s-r}\right)\\
=&\,(t-s) L\left(\frac{x-y}{t-s}\right)-{1\over 2(t-r)}\left({s-r\over t-s}\right)A^{-1}(x-y)\cdot(x-y)\\
\ &\qquad+{1\over t-r}A^{-1}(x-y)\cdot(y-z)-{1\over 2(t-r)}\left({t-s\over s-r}\right)A^{-1}(y-z)\cdot(y-z)\\
=&\,(t-s) L\left(\frac{x-y}{t-s}\right)-{1\over t-r}L\left(\sqrt{s-r\over t-s}(x-y)-\sqrt{t-s\over s-r}(y-z)\right)\\
\leq&\,(t-s) L\left(\frac{x-y}{t-s}\right)\,.
\end{split}
\end{equation*}
In both cases, we obtain the compatibility condition
$$u(t,x)-u(s,y) \leq (t-s) L\left(\frac{x-y}{t-s}\right)\,.$$
Then, see \cite{L}, the right hand side of (\ref{hopft0}) is the unique viscosity solution of
\begin{equation*}
\left\{\begin{array}{ll}
    v_t(t,x)+H(\nabla v(t,x))=0 & \mbox{ a.e. }(t,x)\in Q_{\bar t}\\
    v(t,x)=u(t,x) & \mbox{ for }(t,x)\in \partial Q_{\bar t}\,.
\end{array}\right.
\end{equation*}
Hence it coincides on $Q_{\bar t}$ with the map $u$.

Now fix $X'=(t',x')\in Q$ and let $R'>0$ be such that $B(X',R')\subset Q$. For any $T>0$ the map $u$ is Lipschitz continuous on $[0,T]\times\overline{\Omega}$, see for example \cite{cs}. Let $l>0$ be a Lipschitz constant for $u$ on the set $[0,t'+R']\times\overline{\Omega}$. Let
\begin{equation}\label{lamb}
\lambda=\min\{A^{-1}z\cdot z\,:\,z\in\partial B(0,1)\}>0
\end{equation}
and let $K>R'$ be such that
\begin{equation}\label{defK}
{\lambda(K-R')^2\over l(1+K+R')}\geq1\,.
\end{equation}
If $\max\{0,t'-{1\over2}\}<\bar t<t'$, then for every $(t,x)\in B_{\bar t}(X',R')$ and $(s,y)\in\partial Q_{\bar t}$ such that $\bar t<s<t$ and $|x'-y|\geq K$, by (\ref{lamb}) we obtain
\begin{equation*}
(t-s)L\left({x-y\over t-s}\right)\geq{\lambda|x-y|^2\over2(t-s)}>\lambda(K-R')^2\,.
\end{equation*}
On the other hand, by (\ref{defK}) we have
$$|u(t,x)-u(s,y)|<l(1+K+R')\leq\lambda(K-R')^2\,.$$
Therefore, if $\max\{0,t'-{1\over2}\}<\bar t<t'$ and $(t,x)\in B_{\bar t}(X',R')$, the minimum in (\ref{hopft0}) is realized at some $(s,y)\in B(X',K)\cap \overline{Q}$, such that $(s,y)\in\partial Q_{\bar t}$ and $s<t$. Let
$$M=\sup\left\{\,|u(X)|\,:\,X\in B(X',K)\cap Q\right\}\,.$$
If $0<R<R'$ and $\max\{0,t'-{1\over2}\}<\bar t<t'$ satisfy
$${(d_{\partial\Omega}(x')-R)^2\over t'-\bar t+R}>{4M\over\lambda},$$
then for every $(t,x)\in B_{\bar t}(X',R)$, $\bar t<s<t$ and $y\in\partial\Omega$ we have
\begin{equation*}
\begin{split}
(t-s)\,L\left({x-y\over t-s}\right)+u(s,y)=&{1\over 2(t-s)}A^{-1}(x-y)\cdot(x-y)+u(s,y)\\
\geq&{1\over 2(t-s)}\lambda(d_{\partial\Omega}(x')-R)^2+u(s,y)>M\,.
\end{split}
\end{equation*}
We just proved that if $(t,x)\in B_{\bar t}(X',R)$, the minimum in (\ref{hopft0}) is realized at some point of the form $(\bar t,y)$, $y\in\Omega$. Therefore we obtain (\ref{t0}). 
\end{proof}

\bigskip
In \cite{acns}, the invariance under the generalized gradient flow of the singular set of a solution $u:\R^m\to\R$ of the eikonal equation is proved. The argument of the proof relies on the monotonicity estimate
$$\l u(X)P-u(Y)Q,X-Y\rangle\leq |X-Y|^2$$
for every $P\in D^+u(X)$, $Q\in D^+u(Y)$ and $X,Y\in\R^m$. This property is a direct consequence of (\ref{monsup}) and the global semiconcavity of the square of $u$ with constant $C=2$. Consider now the viscosity solution $u$ of (\ref{syst}). Lemma \ref{lm1} can be exploited in order to obtain some semiconcavity estimates for $u$. For example, fixed $X'=(t',x')\in Q$, let $R>0$ and $0<\bar t<t'$ be associated to $X'$ as in Lemma \ref{lm1}. Using (\ref{t0}), it is possible to verify that for every $x,h\in\R^n$ such that $x-h,x+h\in B(x',R)$,
$$u(t',x+h)-u(t',x-h)-2u(t',x)\leq{\Lambda\over t'-\bar t}|h|^2\,,$$
where $\Lambda=\max\{A^{-1}z\cdot z\,:\,z\in\partial B(0,1)\}$. This semiconcavity property and (\ref{monsup}) imply
$$\l p-q,x-y\r\leq {\Lambda|x-y|^2\over t'-\bar t}$$
for every $p\in \nabla^+u(t',x)$, $q\in \nabla^+u(t',y)$ and $x,y\in B(x',R)$.
Yet, in order to study the propagation in time of the singularities of $u$, we need an estimate on the monotonicity of the superdifferential jointly in time and space. For this reason, as in \cite{acns}, a suitable transform of the solution $u$ is introduced: for any $\bar t>0$, the map $v_{\bar t}:(\bar t,+\infty)\times\Omega\to\R$ is defined by
\begin{equation}\label{defv}
v_{\bar t}(t,x)=(t-\bar t)\,u(t,x)\,.
\end{equation}
In a similar way as before, given $X'=(t',x')\in Q$, for suitable $0<\bar t<t'$ and $R>0$ some semiconcavity properties of $v_{\bar t}$ on $B_{\bar t}(X',R)$ can be obtained, jointly in time and space. In the following lemma, we derive the resulting sharp monotonicity estimate for the superdifferential of $v_{\bar t}$.

\begin{lemma}\label{lm}
Let $X'=(t',x')\in Q$ and let $R>0$ and $0<\bar t<t'$ be associated to $X'$ as in Lemma \ref{lm1}. Then
$$\left\l P_1-P_2,X_1-X_2\right\r\leq2L(x_1-x_2)$$
for every $X_1,X_2\in B_{\bar t}(X',R)$ and every $P_i\in D^+v_{\bar t}(X_i)$, $i=1,2$.
\end{lemma}

\begin{proof} Fix $X'=(t',x')\in Q$ and let $R$ and $\bar t$ be such that (\ref{t0}) holds true for all $(t,x)\in B_{\bar t}(X',R)$. Let $(t,x)\in B_{\bar t}(X',R)\setminus\Sigma(u)$. By (\ref{t0}) there exists $y\in\Omega$ such that
$$u(t,x)=(t-\bar t)L\left({x-y\over t-\bar t}\right)+u(\bar t,y)\,.$$
Then
\begin{equation}\label{tibch}
v_{\bar t}(t,x)=L\left(x-y\right)+(t-\bar t)u(\bar t,y)
\end{equation}
and it is easy to prove that $Dv_{\bar t}(t,x)$ coincides with the gradient of the right hand side of (\ref{tibch}) at $(t,x)$, that is
\begin{equation}\label{fkbvh}
Dv_{\bar t}(t,x)=\left(u(\bar t,y),\nabla L(x-y)\right)\,.
\end{equation}
In general, when $(t,x)\in B_{\bar t}(X',R)$, any element of the superdifferential $D^+v_{\bar t}(t,x)$ is the convex combination of elements of the form (\ref{fkbvh}), since $v_{\bar t}$ is locally semiconcave. Hence, given $X_i=(t_i,x_i)\in B_{\bar t}(X',R)$ and $P_i\in D^+v_{\bar t}(X_i)$, $i=1,2$, there exist $\lambda_i^k\geq0$ and $y_i^k\in\Omega$ for $i=1,2$ and $k\in\{0,\ldots,n+1\}$, such that $\sum_{k=0}^{n+1}\lambda_i^k=1$,
\begin{equation}\label{frekip}
v_{\bar t}(X_i)=L(x_i-y_i^k)+(t_i-\bar t)\,u(\bar t,y_i^k)
\end{equation}
and
\begin{equation}\label{feajio}
P_i=\sum_{k=0}^{n+1}\lambda_i^k\,(u(\bar t,y_i^k),\nabla L(x_i-y_i^k))\,.
\end{equation}
By (\ref{t0}), (\ref{defv}) and (\ref{frekip}), for every $k_1,k_2\in\{0,\ldots,n+1\}$ we obtain
\begin{equation}\label{inneq1}
L(x_1-y_1^{k_1})+(t_1-\bar t)\,u(\bar t,y_1^{k_1})\leq L(x_1-y_2^{k_2})+(t_1-\bar t)\,u(\bar t,y_2^{k_2})
\end{equation}
and
\begin{equation}\label{inneq2}
L(x_2-y_2^{k_2})+(t_2-\bar t)\,u(\bar t,y_2^{k_2})\leq L(x_2-y_1^{k_1})+(t_2-\bar t)\,u(\bar t,y_1^{k_1})\,.
\end{equation}
Thus, (\ref{feajio}), (\ref{inneq1}) and (\ref{inneq2}) imply
\begin{equation}\label{biggg}
\begin{split}
&\left\l P_1-P_2,X_1-X_2\right\r\\
=&\sum_{k_1,k_2=0}^{n+1}\lambda_1^{k_1}\lambda_2^{k_2}\left[(u(\bar t,y_1^{k_1})-u(\bar t,y_2^{k_2}))(t_1-t_2)+\l\nabla L(x_1-y_1^{k_1})-\nabla L(x_2-y_2^{k_2}),x_1-x_2\r\right]\\
=&\sum_{k_1,k_2=0}^{n+1}\lambda_1^{k_1}\lambda_2^{k_2}\left[(t_1-\bar t)(u(\bar t,y_1^{k_1})-u(\bar t,y_2^{k_2}))+(t_2-\bar t)(u(\bar t,y_2^{k_2})-u(\bar t,y_1^{k_1}))\right.\\
\ &\qquad\left.+\,\l\nabla L(x_1-y_1^{k_1})-\nabla L(x_2-y_2^{k_2}),x_1-x_2\r\right]\\
\leq&\sum_{k_1,k_2=0}^{n+1}\lambda_1^{k_1}\lambda_2^{k_2}\left[L(x_1-y_2^{k_2})-L(x_1-y_1^{k_1})+L(x_2-y_1^{k_1})-L(x_2-y_2^{k_2})\right.\\
\ &\qquad\left.+\,\l\nabla L(x_1-y_1^{k_1})-\nabla L(x_2-y_2^{k_2}),x_1-x_2\r\right]\,.
\end{split}
\end{equation}
Observe that the special structure of $L$ yield
\begin{equation}\label{L-L1}
\begin{split}
L(x_1-y_2^{k_2})-L(x_2-y_2^{k_2})&={1\over2}A^{-1}(x_1-x_2)\cdot(x_1-x_2)+A^{-1}(x_1-x_2)\cdot(x_2-y_2^{k_2})\\
&=L(x_1-x_2)+\l\nabla L(x_2-y_2^{k_2}),x_1-x_2\r\,.
\end{split}
\end{equation}
Analogously,
\begin{equation}\label{L-L2}
L(x_2-y_1^{k_1})-L(x_1-y_1^{k_1})=L(x_1-x_2)-\l\nabla L(x_1-y_1^{k_1}),x_1-x_2\r\,.
\end{equation}
Then, combining (\ref{biggg}), (\ref{L-L1}) and (\ref{L-L2}), we obtain
$$\left\l P_1-P_2,X_1-X_2\right\r\leq2L(x_1-x_2)$$
concluding the proof. 
\end{proof}

\section{Propagation of singularities}

In what follows, we denote by $F$ the full Hamiltonian associated with (\ref{syst}), that is, for every $(\tau,p)\in\R\times\R^n$ we set 
$$F(\tau,p)=\tau+H(p).$$

Let $u$ be given by the Hopf formula (\ref{hopf}). Then, see \cite{cs}, $u$ satisfies (\ref{intro:E}) at any point $(t,x)\in Q\setminus\Sigma(u)$. Consequently, for every $X\in Q$ and any $(\tau,p)\in D^*u(X)$, we have $F(\tau,p)=0$. Since $D^+u(X)$ is the convex hull of $D^*u(X)$, the special structure of $F$ implies that $X\in Q$ is a singular point of $u$ if and only if
\begin{equation}\label{minen}
\min_{(\tau,p)\in D^+u(X)}F(\tau,p)
\end{equation}
is strictly negative.
Sufficient conditions are provided in the literature, see for example \cite{CY} and \cite{s}, for the existence of generalized characteristics, whose dynamics are determined by selections of the superdifferentials of $u$ that are "energy minimizing" in the sense of (\ref{minen}). To be more precise, in \cite{CY} it was proved that if for any $X_0=(t_0,x_0)\in Q$ there is a unique generalized characteristic starting from $X_0$, then any generalized characteristic $\xi:[t_0,t_0+T_0)\to Q$ admits right derivative ${\d\over\d s^+}\xi(s)$ for all $s\in[t_0,t_0+T_0)$, this is right-continuous and is given by
\begin{equation*}
{\d\over\d s^+}\xi(s)=DF(\tau(s),p(s))\,,
\end{equation*}
where $(\tau(s),p(s))\in D^+u(\xi(s))$ is such that
\begin{equation}\label{minF}
F(\tau(s),p(s))=\min_{(\tau,p)\in D^+u(\xi(s))}F(\tau,p)\,.
\end{equation}
When $H$ is a quadratic form, the uniqueness of the generalized characteristics, given the initial data, is a consequence of Gronwall's Lemma. Then, we can state the following
\begin{proposition}\label{prop}
Let $X'\in Q$ and $R>0$ be such that $\overline {B(X',R)}\subset Q$. Then, there exists $T_R>0$ such that for every $(t_0,x_0)\in B(X',R)$ there is a Lipschitz continuous arc $\xi:[t_0,t_0+T_R)\to Q$ satisfying the following properties:
\begin{enumerate}
\item[(i)] $\xi(t_0)=(t_0,x_0)$;
\item[(ii)] the right derivative ${\d\over\d s^+}\xi(s)$ does exist for all $s\in[t_0,t_0+T_R)$;
\item[(iii)] ${\d\over\d s^+}\xi(\cdot)$ is right-continuous and
\begin{equation}\label{hgrifj}
{\d\over\d s^+}\xi(s)=DF(\tau(s),p(s))=\left(\begin{array}{cc}1\\\nabla H(p(s))\end{array}\right),
\end{equation}
where $(\tau(s),p(s))\in D^+u(\xi(s))$ satisfies (\ref{minF}).
\end{enumerate}
Moreover,
\begin{equation}\label{dirder}
{\d\over\d s^+}u(\xi(s))=\tau(s)+A\,p(s)\cdot p(s)\qquad\forall\,s\in[t_0,t_0+T_R)\,.
\end{equation}
\end{proposition}



\begin{proof}
The first statement of the Proposition is a consequence of \cite[Theorem 3.2, Corollary 3.4]{CY} and of Gronwall's Lemma. It remains to verify (\ref{dirder}). First observe that by (\ref{hgrifj}) for every $s\in[t_0,t_0+T_R)$ we have
\begin{equation}\label{almostfin}
{\d\over\d s^+}u(\xi(s))=\partial u\left(\xi(s),{\d\over\d s^+}\xi(s)\right)=\partial u\left(\xi(s),\left(\begin{array}{cc}1\\\nabla H(p(s))\end{array}\right)\right)\,.
\end{equation}
Let $(\tau(s),p(s))\in D^+u(\xi(s))$ satisfy (\ref{minF}) for every $s\in[t_0,t_0+T_R)$. By {\it(iii)} of Proposition \ref{prop} we obtain
$$\lim_{h\downarrow0}\,\xi(s+h)=\xi(s)\qquad\lim_{h\downarrow0}\,(\tau(s+h),p(s+h))=(\tau(s),p(s))$$
and
$$\lim_{h\downarrow0}\,{\xi(s+h)-\xi(s)\over|\xi(s+h)-\xi(s)|}={(1,\nabla H(p(s)))\over|(1,\nabla H(p(s)))|}\,.$$
Then Proposition \ref{exposed} yields
$$(\tau(s),p(s))\in D^+u\left(\xi(s),{(1,\nabla H(p(s)))\over|(1,\nabla H(p(s)))|}\right)\,.$$
Equivalently,
\begin{equation}\label{endend}
(\tau(s),p(s))\in{\arg\min}_{(\tau,p)\in D^+u(\xi(s))}\left\l\left(\begin{array}{cc}\tau\\p\end{array}\right),\left(\begin{array}{cc}1\\\nabla H(p(s))\end{array}\right)\right\r\,.
\end{equation}
Therefore, (\ref{almostfin}), (\ref{connect}), (\ref{endend}) and the special structure of $H$ yield
$${\d\over\d s^+}u(\xi(s))=\tau(s)+p(s)\cdot\nabla H(p(s))=\tau(s)+Ap(s)\cdot p(s)\,.$$
\end{proof}

In \cite{CY}, it was shown that, given a singular point $X_0=(t_0,x_0)$ of $u$, the singularity propagates locally in time following the generalized characteristic
\begin{equation}\label{chch}
\xi(s)=\left(\begin{array}{cc}t_0\\x_0\end{array}\right)+\int_{t_0}^s\left(\begin{array}{cc}1\\ \nabla H(p(r))\end{array}\right)\d r\,.
\end{equation}
We shall provide an upper bound for the dissipation of the minimal energy (\ref{minF}) along this curve. Eventually, this estimate has the consequence that the singularities cannot actually extinguish in a finite time.

\begin{theorem}\label{th}
Let $X'=(t',x')\in Q$ and let $R>0$ and $0<\bar t<t'$ be associated to $X'$ as in Lemma \ref{lm1}. Then, there exists $T'>0$ such that for every $(t_0,x_0)\in B_{\bar t}(X',{R\over2})$ the Lipschitz continuous arc defined by (\ref{chch}) satisfies
\begin{equation}\label{estmin}
\min_{(\tau,p)\in D^+u(\xi(s))}F(\tau,p)\leq\left(t_0-\bar t\over s-\bar t\right)^2\min_{(\tau_0,p_0)\in D^+u(t_0,x_0)}F(\tau_0,p_0)
\end{equation}
for every $s\in[t_0,t_0+T')$.
\end{theorem}

Theorem \ref{th} implies the global propagation of the singularities.

\begin{theorem}\label{th:persistence}
Let $(t_0,x_0)$ be a singular point of $u$. Then there exist $T\in(0,+\infty]$ and a Lipschitz continuous arc $\gamma:[t_0,t_0+T)\to\R^n$ starting from $x_0$, satisfying
$$(s,\gamma(s))\in\Sigma(u)\qquad\forall\,s\in[t_0,t_0+T)$$
and such that $\lim_{s\to t_0+T}\gamma(s)\in\partial \Omega$ whenever $T<+\infty$.
\end{theorem}

\begin{proof}
Let $\xi(\cdot)=:(\cdot,\gamma(\cdot))$ be the generalized characteristic starting from $(t_0,x_0)$. Set
\begin{equation}\label{defTmax}
T=\sup\{r\geq0\,:\,\xi(t_0+r)\in\Sigma(u)\}.
\end{equation}
By Theorem \ref{th} we have that $T>0$. Then either $T=+\infty$ or $0<T<+\infty$. In the latter case we must have $x':=\lim_{s\to t_0+T}\gamma(s)\in\partial\Omega$. Indeed, suppose by contradiction that $x'\in\Omega$. Let $R>0$ and $0<\bar t<t_0+T$ be associated with the point $(t',x'):=(t_0+T,x')\in Q$  as in Lemma \ref{lm1}. Let $T'>0$ be as provided by Theorem \ref{th} and let $\max\{\bar t,t'-T'\}<s<t'$ be such that $\xi(s)\in B(\xi(t'),{R\over2})$. Hence, by Theorem \ref{th} we have
$$\min_{(\tau,p)\in D^+u(\xi(t'))}F(\tau,p)\leq\left({s-\bar t\over t'-\bar t}\right)^2\min_{(\tau,p)\in D^+u\left(\xi\left(s\right)\right)}F(\tau,p)<0.$$
Then, $\xi(t_0+T)\in\Sigma(u)\cap Q$ and Theorem \ref{th} contradicts the maximality in (\ref{defTmax}).
\end{proof}

Before proving Theorem \ref{th}, we provide an example showing that the estimate (\ref{estmin}) is somehow sharp.

\noindent
{\bf Example.} For $\varepsilon>0$ consider the problem
\begin{equation}\label{syst22}
\left\{\begin{array}{ll}
    u_t(t,x)+{1\over2}u_x^2(t,x)=0 & \mbox{ a.e. }(t,x)\in (0,+\infty)\times\R\\
    u(0,x)={(|x|-1)^2\over2\varepsilon} ,
\end{array}\right.
\end{equation}
The Hopf formula provides the unique viscosity solution of (\ref{syst22}):
$$u_\varepsilon(t,x)={1\over2}{(|x|-1)^2\over t+\varepsilon}\,.$$
The singular set of this map is $\Sigma(u_\varepsilon)=(0,+\infty)\times\{0\}$ and the curve $\xi:(0,+\infty)\to\R^2$ defined by $\xi(s)=(s,0)$ is a generalized characteristic. If we compute the energy minimizing selection of the superdifferential of $u_\varepsilon$ along this curve, we obtain
$${\arg\min}\{F(\tau,p)\,:\,(\tau,p)\in D^+u_\varepsilon(s,0)\}=\left\{\left(-{1\over2(s+\varepsilon)^2},0\right)\right\}=:\{(\tau(s),p(s))\}\,.$$
Then, given $0<t_0\leq s$, we have
$$F(\tau(s),p(s))=\left({t_0+\varepsilon\over s+\varepsilon}\right)^2F(\tau(t_0),p(t_0))\,.$$
Letting $\varepsilon\downarrow0$ and considering the associated functions $u_\varepsilon$, the inequality (\ref{estmin}) turns out to be sharp.

\begin{proof}[Proof of Theorem \ref{th}]
Fix $X'=(t',x')\in Q$. Let $R>0$ and $0<\bar t<t'$ be associated to $X'$ as in Lemma \ref{lm1}, and $T_R$ be associated to $X'$ and $R$ as in Proposition \ref{prop}. If $(t_0,x_0)\in B_{\bar t}\left(X',{R\over2}\right)$, let $\xi:[t_0,t_0+T_R)\to Q$ be the generalized characteristic starting from $(t_0,x_0)$ and, for any $s\in[t_0,t_0+T_R]$, the vector $(\tau(s),p(s))\in D^+u(\xi(s))$ satisfy
$$F(\tau(s),p(s))\leq F(\tau,p)\qquad\forall\,(\tau,p)\in D^+u(\xi(s)).$$
Observe that, by (\ref{hgrifj}), for every $(t_0,x_0)\in B_{\bar t}\left(X',{R\over2}\right)$ the curve $\xi$ is Lipschitz continuous with constant $l_\xi\leq1+\Lambda l_u$, where $\Lambda=\max\{A^{-1}z\cdot z\,:\,z\in\partial B(0,1)\}$ and $l_u$ is a Lipschitz constant of $u$ on the set $\left[0,t'+{R\over2}+T_R\right]\times\overline{\Omega}$. Setting $T'={R\over 2l_\xi}$, for every $(t_0,x_0)\in B_{\bar t}\left(X',{R\over2}\right)$ we have that
$$\xi(s)\in B_{\bar t}(X',R)\qquad\forall\,s\in[t_0,t_0+T')\,.$$
In order to obtain the claim of the Theorem, we need to verify that
$$F(\tau(s),p(s))\leq\left(t_0-\bar t\over s-\bar t\right)^2 F(\tau(t_0),p(t_0))\qquad\forall\,s\in[t_0,t_0+T')\,.$$

Fix $(t_0,x_0)\in B_{\bar t}\left(X',{R\over2}\right)$.
Recalling \eqref{hgrifj}, for all $s\in[t_0,t_0+T']$ we can define the arc $\gamma$ by
$$\xi(s)=\left(\begin{array}{cc}t_0\\x_0\end{array}\right)+\int_{t_0}^s\left(\begin{array}{cc}1\\ \nabla H(p(r))\end{array}\right)\d r=:\left(\begin{array}{cc}s\\ \gamma(s)\end{array}\right).$$
Let $v_{\bar t}$ be defined as in (\ref{defv}) and fix $s\in(t_0,t_0+T')$. Since $\xi([t_0,t_0+T'])\subset B_{\bar t}(X',R)$, by Lemma \ref{lm}, for every $r\in[t_0,s]$ and $0<h<t_0+T'-s$, we have
\begin{equation}\label{ineq}
\left\l P_h-P,\xi(r+h)-\xi(r)\right\r\leq2L(\gamma(r+h)-\gamma(r))
\end{equation}
for any $P_h\in D^+v_{\bar t}(\xi(r+h)),P\in D^+v_{\bar t}(\xi(r))$.
Let $(t,x)\in B_{\bar t}(X',R)\setminus\Sigma(u)$. By Lemma \ref{lm}, there is $y\in\Omega$ such that
$$u(t,x)=(t-\bar t)L\left({x-y\over t-\bar t}\right)+u(\bar t,y)\,.$$
Then one can prove that
$$Du(t,x)=\left(\begin{array}{cc}-L\left({x-y\over t-\bar t}\right)\\{1\over t-\bar t}\nabla L(x-y)\end{array}\right)\qquad\text{and}\qquad Dv_{\bar t}(t,x)=\left(\begin{array}{cc}u(\bar t,y)\\ \nabla L(x-y)\end{array}\right)\,,$$
so that
$$D v_{\bar t}(t,x)=\left(\begin{array}{cc}u(t,x)\\0\end{array}\right)+(t-\bar t)Du(t,x)\,.$$
This and the properties of the superdifferential of semiconcave functions yield that for every $(t,x)\in B_{\bar t}(X',R)$
\begin{equation}\label{verd}
D^+ v_{\bar t}(t,x)=\left(\begin{array}{cc}u(t,x)\\0\end{array}\right)+(t-\bar t)D^+u(t,x)\,.
\end{equation}
By (\ref{ineq}) and (\ref{verd}), we obtain in particular
\begin{equation*}
\begin{split}
\left\l\left(\begin{array}{cc}u(\xi(r+h))\\0\end{array}\right)\right.&\left.+(r+h-\bar t)\left(\begin{array}{cc}\tau(r+h)\\ p(r+h)\end{array}\right)-\left(\begin{array}{cc}u(\xi(r))\\0\end{array}\right)\right.\\
&\left.+(r-\bar t)\left(\begin{array}{cc}\tau(r)\\p(r)\end{array}\right),\left(\begin{array}{cc}h\\\gamma(r+h)-\gamma(r)\end{array}\right)\right\r\leq2L\left(\gamma(r+h)-\gamma(r)\right),
\end{split}
\end{equation*}
that is
\begin{equation*}
\begin{split}
&h\left[u(\xi(r+h))-u(\xi(r))\right]+h(r+h-\bar t)\left[\tau(r+h)-\tau(r)\right]+h^2\tau(r)\\
+&(r+h-\bar t)\left\l p(r+h)-p(r),\gamma(r+h)-\gamma(r)\right\r+h\left\l p(r),\gamma(r+h)-\gamma(r)\right\r\\
\leq&2L\left(\gamma(r+h)-\gamma(r)\right).
\end{split}
\end{equation*}
Dividing by $h^2$ and setting
$$\tau_h(r)={\tau(r+h)-\tau(r)\over h}\,,\quad p_h(r)={p(r+h)-p(r)\over h}\,,\quad \gamma_h(r)={\gamma(r+h)-\gamma(r)\over h}\,$$
we obtain
\begin{equation}\label{bigineq1}
\begin{split}
{u(\xi(r+h))-u(\xi(r))\over h}+&(r+h-\bar t)\tau_h(r)+\tau(r)\\
+&(r+h-\bar t)\left\l p_h(r),\gamma_h(r)\right\r+\left\l p(r),\gamma_h(r)\right\r
\leq2L\left(\gamma_h(r)\right).
\end{split}
\end{equation}
Let $\rho(r)=\int_{t_0}^r\tau(z)\d z$, and $\rho_h(r)={\rho(r+h)-\rho(r)\over h}$. By Proposition \ref{prop} we have
\begin{eqnarray}\label{gyuh}\nonumber
{u(\xi(r+h))-u(\xi(r))\over h}&=&{1\over h}\int_r^{r+h}{\d\over\d z^+}u(\xi(z))\d z\\
&=&\rho_h(r)+{1\over h}\int_r^{r+h}A\,p(z)\cdot p(z)\d z
\end{eqnarray}
and
\begin{equation}\label{fjdio}
\left\l p_h(r),\gamma_h(r)\right\r={\d\over\d r^+}\left({1\over2}A^{-1}\gamma_h(r)\cdot\gamma_h(r)\right).
\end{equation}
By (\ref{bigineq1}), (\ref{gyuh}) and (\ref{fjdio}) we obtain
\begin{eqnarray}\label{biggineq}\nonumber
2\rho_h(r)+&A^{-1}\gamma_h(r)\cdot\gamma_h(r)+(r+h-\bar t){\d\over\d r^+}\left(\rho_h(r)+{1\over2}A^{-1}\gamma_h(r)\cdot\gamma_h(r)\right)\\\nonumber
&\leq A^{-1}\gamma_h(r)\cdot\gamma_h(r)-{1\over h}\int_r^{r+h}A\,p(z)\cdot p(z)\d z\\\nonumber
&+A^{-1}\gamma_h(r)\cdot\gamma_h(r)-\left\l p(r),\gamma_h(r)\right\r\\
&+\rho_h(r)-{\d\over\d r^+}\rho(r)\,.
\end{eqnarray}
Setting
$$G_h(r)=(r+h-\bar t)^2\left(\rho_h(r)+{1\over2}A^{-1}\gamma_h(r)\cdot\gamma_h(r)\right)\,,$$
the left hand side of (\ref{biggineq}) can be rewritten as $(r+h-\bar t)^{-1}{\d\over\d r^+}G_h(r)$. Now consider the right hand side of the inequality. Define
$$\omega_h(r)=\sup_{r\leq z\leq r+h}[|p(z)-p(r)|+|\tau(z)-\tau(r)|]\,.$$
The sequence $\omega_h(\cdot)$ converges pointwise to 0 as $h\to0+$, since $\tau$ and $p$ are right-continuous. We have
\begin{equation}\label{rrr1}
A^{-1}\gamma_h(r)\cdot\gamma_h(r)-{1\over h}\int_r^{r+h}A\,p(z)\cdot p(z)\d z\leq0\,,
\end{equation}
\begin{equation}\label{rrr2}
\left|A^{-1}\gamma_h(r)\cdot\gamma_h(r)-\left\l p(r),\gamma_h(r)\right\r\right|\leq|\gamma_h(r)|\cdot{1\over h}\int_r^{r+h}|p(z)-p(r)|\d z\leq \|Al_u\omega_h(r)
\end{equation}
and
\begin{equation}\label{rrr3}
\left|\rho_h(r)-{\d\over\d r^+}\rho(r)\right|\leq{1\over h}\int_r^{r+h}|\tau(z)-\tau(r)|\d z\leq\omega_h(r)\,.
\end{equation}
Here $l_u>0$ is a suitable Lipschitz constant for $u$.
Summarizing, (\ref{biggineq}), (\ref{rrr1}), (\ref{rrr2}) and (\ref{rrr3}) yield
\begin{equation}\label{biggineq2}
{\d\over\d r^+}G_h(r)\leq C(r+h-\bar t)\,\omega_h(r)
\end{equation}
for some $C>0$ independent on $r$ and $h$.
Integrating both sides of $(\ref{biggineq2})$ on the interval $[t_0,s]$, we obtain 
$$G_h(s)\leq G_h(t_0)+C\int_{t_0}^s(r+h-\bar t)\,\omega_h(r)\,.$$
Taking the limit as $h\to0$ and using dominated convergence, we obtain
$$F(\tau(s),p(s))\leq\left(t_0-\bar t\over s-\bar t\right)^2\,F(\tau(t_0),p(t_0))\,.$$
This concludes the proof of the Theorem.

\end{proof}

\section*{Acknowledgements} This work was co-funded by
the European Union under the 7th Framework Programme
``FP7-PEOPLE-2010-ITN'', grant agreement number
264735-SADCO.


\medskip
Received xxxx 20xx; revised xxxx 20xx.
\medskip

\end{document}